\documentclass[12pt]{article}
\usepackage{amsmath, amsthm, amssymb}
\usepackage{hyperref}
\usepackage{verbatim}
\usepackage[top=1.0in, bottom=1.0in, left=1.0in, right=1.0in]{geometry}

\usepackage{tkz-graph}
\usetikzlibrary{arrows}
\usetikzlibrary{shapes}
\usepackage[position=bottom]{subfig}

\usepackage{longtable}
\usepackage{array}

\usepackage{sectsty}
\allsectionsfont{\sffamily}

\makeatletter
\newtheorem*{rep@theorem}{\rep@title}
\newcommand{\newreptheorem}[2]{
\newenvironment{rep#1}[1]{
 \def\rep@title{#2 \ref{##1}}
 \begin{rep@theorem}}
 {\end{rep@theorem}}}
\makeatother

\theoremstyle{plain}
\newtheorem{thm}{Theorem}[section]
\newreptheorem{thm}{Theorem}

\newreptheorem{prop}{Proposition}
\newtheorem{lem}[thm]{Lemma}
\newreptheorem{lem}{Lemma}
\newtheorem{conjecture}[thm]{Conjecture}
\newreptheorem{conjecture}{Conjecture}
\newtheorem{cor}[thm]{Corollary}
\newreptheorem{cor}{Corollary}

\newtheorem{observation}{Observation}
\newtheorem*{mainconj}{Main Conjecture}
\newtheorem*{mainthm}{Main Theorem}

\theoremstyle{definition}
\newtheorem{defn}{Definition}
\theoremstyle{remark}

\newcommand{\fancy}[1]{\mathcal{#1}}

\newcommand{\set}[1]{\left\{ #1 \right\}}

\newcommand{\card}[1]{\left|#1\right|}

\newcommand{\ceil}[1]{\left\lceil#1\right\rceil}
\newcommand{\floor}[1]{\left\lfloor#1\right\rfloor}

\newcommand{\irange}[1]{\left[#1\right]}

\newcommand{\parens}[1]{\left( #1 \right)}
\newcommand{\brackets}[1]{\left[ #1 \right]}

\newcommand{\DefinedAs}{\mathrel{\mathop:}=}

\def\Q{\fancy{Q}}
\def\Z{\fancy{Z}}

\begin{document}
\title{
A note on coloring vertex-transitive graphs
}
\author{Daniel W. Cranston\thanks{Department of Mathematics and Applied
Mathematics, Virginia Commonwealth University, Richmond, VA, 23284. email:
\texttt{dcranston@vcu.edu}} \and 
Landon Rabern\thanks{Lancaster, PA, 17601.  email:
\texttt{landon.rabern@gmail.com}. 
}}
\maketitle
\date

\begin{abstract}
We prove bounds on the chromatic number $\chi$ of a vertex-transitive graph in
terms of its clique number $\omega$ and maximum degree $\Delta$.  We conjecture
that every vertex-transitive graph satisfies $\chi \le \max \set{\omega,
\ceil{\frac{5\Delta + 3}{6}}}$ and we prove results supporting this conjecture. 
Finally, for vertex-transitive graphs with $\Delta \ge 13$
we prove the Borodin-Kostochka conjecture, i.e.,
$\chi\le\max\set{\omega,\Delta-1}$.
\end{abstract}

\section{Introduction}
Many results and conjectures in the graph coloring literature have the
form: \emph{if the chromatic number $\chi$ of a graph is close to its maximum
degree $\Delta$,
then the graph contains a big clique, i.e., $\omega$ is large}
(\cite{brooks1941colouring,
borodin1977upper, reed1999strengthening, reed1998omega, CatlinAnotherBound,
molloy2002graph}). Generically, we call conjectures of this sort \emph{big
clique conjectures}.  In
\cite{denseneighborhoods}, it was shown that many big clique conjectures 
hold under the added hypothesis that every vertex is in a medium sized clique. 
Partial results on big clique conjectures often guarantee a medium sized clique,
but not a big clique.  But in a vertex-transitive graph, the existence of one
medium sized clique implies that every vertex is in a medium sized clique.
By applying the idea in \cite{denseneighborhoods}, we now get a big clique. 
So, in essence, partial results on big clique conjectures are
self-strengthening in the class of vertex-transitive graphs.

In this short note, we give some examples of this phenomenon.  There is not
much new graph theory here, just combinations of known results that yield facts
we did not know.  The following conjecture is the best we could hope for. 
A good deal of evidence supports it, as we will detail below.

\begin{mainconj}
If $G$ is vertex-transitive, then $\chi(G) \le \max \set{\omega(G), \ceil{\frac{5\Delta(G) + 3}{6}}}$.
\end{mainconj}

Our Main Conjecture would be best possible, as shown by Catlin's
counterexamples to the Haj{\'o}s conjecture \cite{catlin1979hajos}.  Catlin
computed the chromatic number of line graphs of odd cycles where each edge has
been duplicated $k$ times; in particular, he showed that $\chi(G_{t,k}) = 2k +
\ceil{\frac{k}{t}}$ for $t \ge 2$, where $G_{t,k} \DefinedAs L(kC_{2t+1})$.
Since $\Delta(G_{t,k}) = 3k-1$ and $\omega(G_{t,k}) = 2k$, we have
$\chi(G_{2, k}) = 2k+\ceil{\frac{k}2}=\ceil{\frac{5k}2} = 
\ceil{\frac{15k-2}6} =
\max \set{\omega(G_{2,k}), \ceil{\frac{5\Delta(G_{2,k}) + 3}{6}}}$ for
all $k \ge 1$.

The Main Conjecture does not hold for graphs in general.  To see this, let $H_t$ be $K_t$ joined to a $5$-cycle. 
For any $c \ge 2$, we can make $t$ large enough so that $\chi(H_t) > \max\set{\omega(H_t), \Delta(H_t) - c}$ and $\Delta(H_t)$ is as large as we like.  
So, no bound of this form can hold in general, not even for claw-free graphs (since $H_t$ is claw-free). In \cite{rabern2011strengthening}, it was shown that a bound of this form does hold for line graphs of multigraphs.  In particular, they satisfy $\chi \le \max \set{\omega, \frac{7\Delta + 10}{8}}$.  The bound in the Main Conjecture would be best possible for line graphs of multigraphs as well.

Our main result is the following weakening of the Borodin-Kostochka conjecture
for vertex-transitive graphs, which we prove in Section~\ref{BK}.  This theorem
likely holds for all $\Delta \ge 9$ and proving this may be a good deal easier
than proving the full Borodin-Kostochka conjecture (note that the Main
Conjecture implies the Main Theorem for all $\Delta \ge 9$).

\begin{mainthm}\label{BKTransitive}
If $G$ is vertex-transitive with $\Delta(G) \ge 13$ and $K_{\Delta(G)} \not \subseteq G$, then $\chi(G) \le \Delta(G) - 1$.
\end{mainthm}

As further evidence for the Main Conjecture, we show that the analogous upper
bound holds for the fractional chromatic number.  Also, we show that the Main
Conjecture is true if all vertex-transitive graphs satisfy both Reed's
$\omega$, $\Delta$, and $\chi$ conjecture and the strong $2\Delta$-colorability
conjecture (see \cite{aharoni2007independent}; really we can get by with
$\frac52\Delta$-colorability).  Finally, we show the following.

\begin{thm}
There exists $c < 1$, such that for any vertex-transitive graph $G$, we have $\chi(G) \le \max \set{\omega(G), c(\Delta(G) + 1)}$.
\end{thm}

\section{Clustering of maximum cliques}
Before coloring anything, we need a better understanding of the structure of maximum cliques in a graph.
\subsection{The clique graph}
\begin{defn}
Let $G$ be a graph. For a collection of cliques $\Q$ in $G$, let $X_\Q$ be the intersection graph of $\Q$; that is, the vertex set of $X_\Q$ is $\Q$ and there is an edge between $Q_1, Q_2 \in \Q$ iff $Q_1 \ne Q_2$ and $Q_1$ and $Q_2$ intersect.
\end{defn}

When $\Q$ is a collection of maximum cliques, we get a lot of information about $X_\Q$.  Kostochka \cite{kostochkaRussian} used the following lemma of Hajnal \cite{HajnalSaturation} to show that the components of $X_\Q$ are complete in a graph with $\omega > \frac23 (\Delta + 1)$.

\begin{lem}[Hajnal \cite{HajnalSaturation}]\label{HajnalLemma}
If $G$ is a graph and $\Q$ is a collection of maximum cliques in $G$, then 
\[\card{\bigcup \Q} + \card{\bigcap \Q} \geq 2\omega(G).\]
\end{lem}

Hajnal's lemma follows by an easy induction.  The proof of Kostochka's lemma in
\cite{kostochkaRussian} is in Russian; for a reproduction of his original proof
in English, see \cite{rabernhitting}.  Below we give a shorter proof from
\cite{raberndiss}.

\begin{lem}[Kostochka \cite{kostochkaRussian}]\label{KostochkaCliqueGraph}
If $\Q$ is a collection of maximum cliques in a graph $G$ with $\omega(G) > \frac23 (\Delta(G) + 1)$ such that $X_\Q$ is connected, then $\cap \Q \neq \emptyset$. 
\end{lem}
\begin{proof}
Suppose not and choose a counterexample $\Q \DefinedAs \set{Q_1, \ldots, Q_r}$
minimizing $r$. Plainly, $r \geq 3$. Let $A$ be a noncutvertex in $X_{\Q}$ and
$B$ a neighbor of $A$. Put $\fancy{Z} \DefinedAs \Q - \set{A}$. Then
$X_{\fancy{Z}}$ is connected and hence by minimality of $r$, $\cap \fancy{Z}
\neq \emptyset$. In particular, $\card{\cup \fancy{Z}} \leq \Delta(G) + 1$.
By assumption, $\cap\Q=\emptyset$, so $\card{\cap \Q} + \card{\cup \Q} \leq 0 +
(\card{\cup \fancy{Z}} + \card{A - B}) \leq
(\Delta(G) + 1) + (\Delta(G)+1 - \omega(G)) < 2\omega(G)$. This contradicts
Lemma \ref{HajnalLemma}.
\end{proof}

As shown by Christofides, Edwards and King \cite{christofides2012note}, components of $X_\Q$ have nice structure in the $\omega = \frac23 (\Delta + 1)$ case as well.  We'll need this stronger result to get our bounds on coloring vertex-transitive graphs to be tight.

\begin{lem}[Christofides, Edwards and King \cite{christofides2012note}]\label{TwoThirdsEqualityStructure}
If $\Q$ is a collection of maximum cliques in a graph $G$ with $\omega(G) \ge \frac23 (\Delta(G) + 1)$ such that $X_\Q$ is connected, then either 
\begin{itemize}
\item $\cap \Q \ne \emptyset$; or
\item $\Delta(X_\Q) \le 2$ and if $B, C \in \Q$ are different neighbors of $A \in \Q$, then $B \cap C = \emptyset$ and $\card{A \cap B} = \card{A \cap C} = \frac12 \omega(G)$.
\end{itemize}

\end{lem}
\subsection{In vertex-transitive graphs}
Let $G$ be a vertex-transitive graph and let $\Q$ be the collection of all
maximum cliques in $G$.  It is not hard to see that $X_\Q$ is vertex-transitive
as well; in fact, we have the following.

\begin{observation}\label{transitiveClustering}
Let $G$ be a vertex-transitive graph and let $\Q$ be the collection of all
maximum cliques in $G$.  For each component $C$ of $X_\Q$, put $G_C \DefinedAs
G\brackets{\bigcup V(C)}$.  Then $G_C$ is vertex-transitive for each component
$C$ of $X_\Q$ and $G_{C_1} \cong G_{C_2}$ for components $C_1$ and $C_2$ of
$X_\Q$.
\end{observation}

A basic consequence of Observation \ref{transitiveClustering} is that if $G$ is
vertex-transitive and $G_C$ has a dominating vertex (or universal vertex), then
every vertex of $G_C$ is dominating; so $G_C$ is complete.  
Let $G$ be a vertex-transitive graph with $\omega > \frac23 (\Delta + 1)$.   
Suppose that $X_\Q$ has one or more edges.  By Kostochka's lemma, $\cap\Q_C$ is
nonempty, where $\Q_C$ is the set of maximum cliques in some component $G_C$.
Choose a vertex $v\in\cap\Q_C$, and note that $v$ is adjacent to each vertex
in $G_C$.  Since $G$ is vertex-transitive, each vertex of $G_C$ is a
dominating vertex in $G_C$; so, in fact, $G_C$ is a clique, and $C$ is
edgeless.
Using Lemma \ref{TwoThirdsEqualityStructure}, we get a bit more.

\begin{lem}\label{TransitiveClusteringBigCliques}
Let $G$ be a connected vertex-transitive graph and let $\Q$ be the collection
of all maximum cliques in $G$.  If $\omega(G) \ge \frac23 \parens{\Delta(G) +
1}$, then either
\begin{itemize}
\item $X_\Q$ is edgeless; or
\item $X_\Q$ is a cycle and $G$ is the graph obtained from $X_\Q$ by blowing up each vertex to a $K_{\frac12 \omega(G)}$.
\end{itemize}
\end{lem}
\begin{proof}
If $\omega(G) > \frac23 \parens{\Delta(G) + 1}$, then $X_\Q$ is edgeless as shown above.  Hence we may assume $\omega(G) = \frac23 \parens{\Delta(G) + 1}$.  Let $Z$ be a component of $X_\Q$ and put $\Z \DefinedAs V(Z)$.
By Lemma \ref{TwoThirdsEqualityStructure}, $\Delta(X_\Z) \le 2$ and if $B, C \in \Z$ are different neighbors of $A \in \Z$, then $B \cap C = \emptyset$ and $\card{A \cap B} = \card{A \cap C} = \frac12 \omega(G)$.  By Observation \ref{transitiveClustering}, $X_\Z$ must be a cycle.  But then every vertex in $G_Z$ has $\frac12 \omega(G) + \frac12 \omega(G) + \frac12 \omega(G) - 1 = \Delta(G)$ neighbors in $G_Z$ and thus $G = G_Z$.  Hence $X_\Q = X_\Z$ is a cycle and $G$ is the graph obtained from $X_\Q$ by blowing up each vertex to a $K_{\frac12 \omega(G)}$.
\end{proof}

\section{The fractional version}
The problem of determining chromatic number can be phrased as an integer
program: we aim to minimize the total number of colors used, subject to the
constraints that (i) each vertex gets colored and (ii) the vertices receiving
each color form an independent set.  To reach a linear program from this
integer program, we relax the constraint that each vertex is colored with a
single color, and instead allow a vertex to be colored with a combination of
colors, e.g., $1/2$ red, $1/3$ green, and $1/6$ blue. However, we still
require that the total weight of any color on any clique is at most 1.
The minimum value of this linear program is the fractional chromatic number,
denoted $\chi_f$ (see~\cite{ScheinermanUllmanFrac} for a formal definition and
many results on fractional coloring).

It is an easy exercise to show that every vertex-transitive graph $G$ satisfies
$\chi_f(G) = \frac{|G|}{\alpha(G)}$, where $|G|$ denotes $|V(G)|$ and
$\alpha(G)$ denotes the maximum size of an independent set. We also need
Haxell's condition \cite{haxell2001note} for the
existence of an independent transversal.

\begin{lem}[Haxell \cite{haxell2001note}]
\label{HaxellTransversal}
Let $H$ be a graph and $V_1 \cup \cdots \cup V_r$ a partition of $V(H)$. 
Suppose that $\card{V_i} \geq 2\Delta(H)$ for each $i \in \irange{r}$. Then $H$ has an independent set $\set{v_1, \ldots, v_r}$
where $v_i \in V_i$ for each $i \in \irange{r}$.
\end{lem}

\begin{lem}\label{TransitiveFractionalColoringWithBigCliques}
If $G$ is a vertex-transitive graph with $\omega(G) \ge \frac23 \parens{\Delta(G) + 1}$, then $\alpha(G) = \floor{\frac{|G|}{\omega(G)}}$.  Moreover, if $\omega(G) > \frac23 \parens{\Delta(G) + 1}$, then $\omega(G)$ divides $|G|$.
\end{lem}
\begin{proof}
We may assume that $G$ is connected. Since $G$ is vertex-transitive, every
vertex of $G$ is in an $\omega(G)$-clique. First, suppose $\omega(G) > \frac23
\parens{\Delta(G) + 1}$.  Then Lemma \ref{TransitiveClusteringBigCliques} shows
that the vertex set of $G$ can be partitioned into cliques $V_1, \ldots, V_r$
with $|V_i| \ge \ceil{\frac23 \parens{\Delta(G) + 1}}$ for each $i \in
\irange{r}$.  Let $H$ be the graph formed from $G$ by making each $V_i$
independent.  Then $\Delta(H) \le \Delta(G) + 1 - \ceil{\frac23
\parens{\Delta(G) + 1}}$; now by Lemma \ref{HaxellTransversal}, $G$ has
an independent set with a vertex in each $V_i$.  Since $G$ is vertex-transitive,
all $V_i$ have the same size; so, in fact, $\card{V_i}=\omega(G)$ for all $i$.
But now $\card{G} = \alpha(G)\card{V_i}=\alpha(G)\omega(G)$, so we're done.

So instead suppose $\omega(G) = \frac23 \parens{\Delta(G) + 1}$.  Now Lemma
\ref{TransitiveClusteringBigCliques} shows that $G$ is obtained from a cycle
$C$ by blowing up each vertex of $C$ to a copy of $K_{\frac12 \omega(G)}$. 
Hence $\alpha(G) = \floor{\frac{|C|}{2}} = \floor{\frac{|G|}{\omega(G)}}$ as
desired.
\end{proof}

Reed's $\omega$, $\Delta$, and $\chi$ conjecture states that every graph
satisfies 
\[\chi \leq\ceil{\frac{\omega + \Delta + 1}{2}}.\]

In \cite{molloy2002graph}, Molloy and Reed proved this upper bound without the
round-up for the fractional chromatic number $\chi_f$.  
Since $\chi_f(G) = \frac{|G|}{\alpha(G)}$ for vertex-transitive graphs, an
earlier result of Fajtlowicz \cite{fajtlowicz1984independence} suffices for our
purposes.

\begin{lem}[Fajtlowicz \cite{fajtlowicz1984independence}]\label{fajtlowicz}
For every graph $G$, we have $\alpha(G) \ge \frac{2|G|}{\omega(G) + \Delta(G) + 1}$.
\end{lem}

\begin{thm}
If $G$ is vertex-transitive, then $\alpha(G) \ge \frac{|G|}{\max\set{\omega(G), \frac56\parens{\Delta(G) + 1}}}$.
\label{fracversion}
\end{thm}
\begin{proof}
Suppose $\omega(G) > \frac23 \parens{\Delta(G) + 1}$.  Then Lemma \ref{TransitiveFractionalColoringWithBigCliques} shows $\alpha(G) = \frac{|G|}{\omega(G)}$ and we're done.  Otherwise, $\omega(G) \le \frac23 \parens{\Delta(G) + 1}$ and Lemma \ref{fajtlowicz} gives $\alpha(G) \ge \frac{2|G|}{\frac23 (\Delta(G) + 1) + \Delta(G) + 1} = \frac{|G|}{\frac56 (\Delta(G) + 1)}$ as desired.
\end{proof}

Restating Theorem~\ref{fracversion} in terms of fractional coloring, we have
the following.

\begin{cor}
If $G$ is vertex-transitive, then $\chi_f(G) \le \max\set{\omega(G), \frac56\parens{\Delta(G) + 1}}$.
\end{cor}

\section{Reed's conjecture plus strong coloring}
For a positive integer $r$, a graph $G$ with $\card{G} = rk$ is called
\emph{strongly $r$-colorable} if for every partition of $V(G)$ into parts of
size $r$ there is a proper coloring of $G$ that uses all $r$ colors on each
part.  If $\card{G}$ is not a multiple of $r$, then $G$ is strongly
$r$-colorable iff the graph formed by adding $r\ceil{\frac{|G|}{r}} - |G|$
isolated vertices to $G$ is strongly $r$-colorable.  The \emph{strong chromatic
number} $s\chi(G)$ is the smallest $r$ for which $G$ is strongly $r$-colorable.
Not surprisingly, if $G$ is strongly $r$-colorable, then $G$ is also strongly
$(r+1)$-colorable, although the proof of this fact is
non-trivial~\cite{Fellows1990}.

In \cite{haxell2004strong}, Haxell proved that the strong chromatic number of
any graph is at most $3\Delta - 1$.  
In \cite{haxell2008strong}, she proved further that for every $c>11/4$ there
exists $\Delta_c$ such that if $G$ has maximum degree $\Delta$ at least
$\Delta_c$, then $G$ has strong chromatic number at most $c\Delta$.
The strong $2\Delta$-colorability conjecture \cite{aharoni2007independent} says
that the strong chromatic number of any graph is at most $2\Delta$.  If true,
this conjecture would be sharp. We need the following intermediate
conjecture.

\begin{conjecture}\label{StrongTransitive}
The strong chromatic number of any vertex-transitive graph is at most $\frac52 \Delta$.
\end{conjecture}

We also need Reed's conjecture \cite{reed1998omega} restricted to vertex-transitive graphs.

\begin{conjecture}\label{ReedTransitive}
Every vertex-transitive graph satisfies $\chi \leq\ceil{\frac{\omega + \Delta + 1}{2}}$.
\end{conjecture}

\begin{thm}\label{ConjectureaImplyConjecture}
If Conjecture \ref{StrongTransitive} and Conjecture \ref{ReedTransitive} both
hold, then the Main Conjecture does as well.
\end{thm}
\begin{proof}
We may assume that $G$ is connected. Put $\Delta \DefinedAs \Delta(G)$, $\omega
\DefinedAs \omega(G)$ and $\chi \DefinedAs \chi(G)$. Suppose $\omega < \frac23
\parens{\Delta + 1}$.  So, we have $\omega \le \frac{2\Delta + 1}{3}$ and
moreover, when $\Delta \equiv 3 \text{ (mod $6$)}$, we have $\omega \le \frac23
\Delta$.  Plugging the first inequality into Conjecture \ref{ReedTransitive}
gives $\chi \le \ceil{\frac{5\Delta + 4}{6}} = \ceil{\frac{5\Delta + 3}{6}}$
when $\Delta \not \equiv 3 \text{ (mod $6$)}$; by using the improved upper
bound on $\omega$ in the remaining case, we again prove the desired upper bound
on $\chi$.

Now suppose $\omega \ge \frac23 (\Delta + 1)$ and let $\fancy{Q}$ be the set of
maximum cliques in $G$.  Applying Lemma \ref{TransitiveClusteringBigCliques},
either $X_\fancy{Q}$ is edgeless or $G$ is obtained from an odd cycle by
blowing up each vertex to a $K_{\frac{\omega}{2}}$.  In the latter case, $G$ is
one of Catlin's examples from \cite{catlin1979hajos} and the bound holds as
mentioned in the introduction.  Hence we may assume that $X_\fancy{Q}$ is
edgeless; that is, $V(G)$ can be partitioned into $\omega(G)$-cliques.

Suppose $\chi > \omega$. Now we show that Conjecture \ref{StrongTransitive}
implies the Main Conjecture.
Form $G'$ from $G$ by adding vertices to the maximum cliques of $G$ until they
all have $\ceil{\frac{5\Delta + 3}{6}}$ vertices; each new vertex has no edges
outside its clique, and $\Delta$ always denotes the maximum degree in $G$, not
in $G'$. 
Now form $G''$ from $G'$ by removing all edges within each maximum clique.  
Each vertex now has at most $\Delta + 1 - \omega \le
\frac13 (\Delta + 1)$ neighbors in $G'$ outside of its clique, 
hence the maximum
degree of $G''$ is at most $\frac13(\Delta+1)$.  
Since $\ceil{\frac{5\Delta + 3}{6}} \ge
\frac52 \parens{\frac13 (\Delta + 1)}$, Conjecture \ref{StrongTransitive}
implies that $G''$ is strongly 
$\ceil{\frac{5\Delta + 3}{6}}$-colorable.
By taking the $V_i$'s of $G''$ to be the vertex sets of the maximum cliques in
$G'$, we see that $G'$ is $\ceil{\frac{5\Delta + 3}{6}}$-colorable, and hence
so is $G$.
\end{proof}

Reed \cite{reed1998omega} has shown that there is $0 < \epsilon < 1$ such that
every graph satisfies $\chi \le \epsilon\omega + (1-\epsilon)(\Delta+1)$; for a
shorter and simpler proof, see~\cite{KingReed2012}.  Combining this upper bound
with Haxell's $3\Delta - 1$ strong colorability result, we get the following
similarly to Theorem \ref{ConjectureaImplyConjecture}.

\begin{thm}
There exists $c < 1$, such that for any vertex-transitive graph $G$, we have $\chi(G) \le \max \set{\omega(G), c(\Delta(G) + 1)}$.
\end{thm}

\section{Borodin-Kostochka for vertex-transitive graphs}
\label{BK}
In \cite{bigcliques}, we proved the following.
\begin{thm}\label{BigCliquesExist}
If $G$ is a graph with $\Delta(G) \ge 13$ and $K_{\Delta(G) - 3} \not \subseteq G$, then $\chi(G) \le \Delta(G) - 1$.
\end{thm} 

In \cite{denseneighborhoods}, the second author proved the following.
\begin{thm}\label{DenseNeighbors}
If $G$ is a graph with $\Delta(G) \ge 9$ and $K_{\Delta(G)} \not \subseteq G$ such that every vertex is in a clique on $\frac23 \Delta(G) + 2$ vertices, then $\chi(G) \le \Delta(G) - 1$.
\end{thm}

By combining these theorems, we immediately get that the Borodin-Kostochka
conjecture holds for vertex-transitive graphs with $\Delta \ge 15$.  We can
improve this result using Lemma \ref{TransitiveClusteringBigCliques} and
Haxell's $3\Delta - 1$ strong colorability result.

\begin{mainthm}\label{BKTransitive}
If $G$ is vertex-transitive with $\Delta(G) \ge 13$ and $K_{\Delta(G)} \not \subseteq G$, then $\chi(G) \le \Delta(G) - 1$.
\end{mainthm}
\begin{proof}
Suppose that $\chi(G)\ge \Delta(G)$.
By Lemma \ref{BigCliquesExist}, we have $\omega(G) \ge \Delta(G) - 3 > \frac23
(\Delta(G) + 1)$ since $\Delta(G) \ge 13$.  Now Lemma
\ref{TransitiveClusteringBigCliques} shows that $X_{\Q}$ is edgeless, where
$\Q$ is the collection of all maximum cliques in $G$.  

Form $G'$ from $G$ by adding vertices to the
maximum cliques of $G$ until they all have $\Delta(G) - 1$ vertices, where each
new vertex has no edges outside its clique.  Each vertex has at most $\Delta(G) + 1 - \omega(G) \le 4$ neighbors outside its clique.  Since $\Delta(G) - 1 = 12 \ge 3*4 - 1$, Haxell's $3\Delta - 1$ strong colorability result implies that $G'$ is $\parens{\Delta(G) - 1}$-colorable and hence so is $G$.
\end{proof}

If Conjecture \ref{ReedTransitive} holds, then we get $\omega \ge \Delta - 2 > \frac23 (\Delta + 1)$ when $\Delta \ge 9$.  So, since $\Delta + 1 - \omega \le 3$, the above argument works for $\Delta \ge 9$.  That is, Conjecture \ref{ReedTransitive} by itself implies the Borodin-Kostochka conjecture for vertex-transitive graphs.

\bibliographystyle{amsplain}
\bibliography{GraphColoring}
\end{document}